\documentclass[12pt]{amsart}
\usepackage{epsfig,color}

\theoremstyle{definition}

\newtheorem{theorem}{Theorem}[section]
\newtheorem{definition}[theorem]{Definition}

\newtheorem{proposition}[theorem]{Proposition}
\newtheorem{lemma}[theorem]{Lemma}
\newtheorem{remark}[theorem]{Remark}

\numberwithin{equation}{section}

\newcommand{\abs}[1]{\lvert#1\rvert}
\newcommand{\dive}{\text{div}}
\newcommand{\ee}{{\text{e}}}
\newcommand{\hps}{{\text{hypersurface}}}
\newcommand{\sk}{{\text{self-shrinker}}}
\newcommand{\sks}{{\text{self-shrinkers}}}
\newcommand{\lb}[1]{\langle#1\rangle}

\newcommand{\mf}{\mathbf}

\DeclareMathOperator*{\vol}{Vol}
\DeclareMathOperator*{\area}{Area}

\title{Rigidity and Curvature Estimates for Graphical Self-shrinkers}

\author{Qiang Guang}%
\address{Department of Mathematics,
Massachusetts Institute of Technology, Cambridge, MA 02139, USA}
\email{qguang@math.mit.edu}

\author{Jonathan J. Zhu}%
\address{Department of Mathematics,
Harvard University, Cambridge, MA 02138, USA}
\email{jjzhu@math.harvard.edu}

\begin{document}

\begin{abstract}
Self-shrinkers are hypersurfaces that shrink homothetically under mean curvature flow; these solitons model the singularities of the flow. It it presently known that an entire self-shrinking graph must be a hyperplane. In this paper we show that the hyperplane is rigid in an even stronger sense, namely: For $2\leq n \leq 6$, any smooth, complete self-shrinker $\Sigma^n\subset\mathbf{R}^{n+1}$ that is graphical inside a large, but compact, set must be a hyperplane. In fact, this rigidity holds within a larger class of almost stable self-shrinkers.%

A key component of this paper is the procurement of linear curvature estimates for almost stable shrinkers, and it is this step that is responsible for the restriction on $n$. Our methods also yield uniform curvature bounds for translating solitons of the mean curvature flow.
\end{abstract}

\maketitle


\section{Introduction}
\label{sec:intro}

Mean curvature flow describes the evolution of a hypersurface moving by its mean curvature vector; such a flow corresponds to the negative gradient flow for surface area.

We will call a hypersurface $\Sigma^n\subset\mathbf{R}^{n+1}$ a \textit{self-shrinker}, or more simply a \textit{shrinker}, if it is the $t=-1$ time-slice of a mean curvature flow that evolves by shrinking homothetically to the origin $x=0$. Such hypersurfaces are characterized by the equation \begin{equation}\label{eq:shrinker}
H = \frac{1}{2}\langle x,\mathbf{n}\rangle,
\end{equation}
where $\mathbf{n}$ is the surface normal and $H=\dive_\Sigma \mathbf{n}$ is the mean curvature at $x$. Self-shrinkers may equivalently be interpreted as critical points for the Gaussian area $\int_\Sigma \ee^{-|x|^2/4}$, hence also as minimal hypersurfaces for the conformal metric $g_{ij} = \ee^{-|x|^2/2n}\delta_{ij}$. By Huisken's monotonicity formula \cite{Hui90},
blow-ups at singularities of a mean curvature flow converge to self-shrinkers. The study of self-shrinkers is thus vital to understanding the mean curvature flow.

When $n=1$, mean curvature flow is also known as curve-shortening flow; in this case, the only complete, smoothly embedded self-shrinkers are lines and the circle of radius $\sqrt{2}$. So henceforth we will assume $n\geq 2$.

In this paper we prove the following strong rigidity theorem for self-shrinkers that are graphical on large balls:
\begin{theorem}
\label{thm:gmain}
Given $ n\leq 6$ and $\lambda_0$, there exists $R=R(n,\lambda_0)$ so that if $\Sigma^n \subset \mf{R}^{n+1}$ is a smooth, complete \sk\ with entropy $\lambda(\Sigma)\leq \lambda_0$ satisfying
 \begin{itemize}
  \item[($\dagger$)] $\Sigma$ is graphical in $B_R$,
  \end{itemize}
  then $\Sigma$ is a hyperplane.
\end{theorem}
\begin{remark}
It is important to note that we do not assume any a priori curvature bounds in the above theorem. Indeed, the conclusion follows readily, and in all dimensions, if one assumes even a mild curvature estimate (see for instance Remark \ref{rmk:bootstrap}).
\end{remark}

For self-shrinkers, the entropy coincides with the Gaussian surface area. Here we take the graphical condition ($\dagger$) to mean that there is a constant vector $v\in\mathbf{R}^{n+1}$ such that the hypersurface normal $\mathbf{n}$ satisfies $\langle v,\mathbf{n}\rangle>0$ at each point of $\Sigma \cap B_R$. This is equivalent to each connected component of $\Sigma \cap B_R$ being a graph over a region in the hyperplane $\langle v,x\rangle=0$.

In fact, we will show that the hyperplane is rigid amongst a larger class of self-shrinkers, namely those that satisfy an `almost stability' on large balls. This constitutes our main theorem, stated as follows:

\begin{theorem}
\label{thm:main}
Given $ n\leq 6$ and $\lambda_0$, there exists $R=R(n,\lambda_0)$ so that if $\Sigma^n \subset \mf{R}^{n+1}$ is a smooth complete \sk\ with entropy $\lambda(\Sigma)\leq \lambda_0$ satisfying
 \begin{itemize}
  \item[($\ddagger$)] $\Sigma$ is $\frac{1}{2}$-stable in $B_R$,
  \end{itemize}
  then $\Sigma$ is a hyperplane.
\end{theorem}

We define a more general term, $\delta$-stability, in Section \ref{sec:stability}; essentially it means that the second variation operator $L$ of the Gaussian area functional has first eigenvalue at least $-\delta$. In particular $\frac{1}{2}$-stability is a weaker assumption than $L$-stability.

Theorem \ref{thm:gmain} follows immediately from Theorem \ref{thm:main} together with the observation that a graphical self-shrinker $\Sigma$ is $\frac{1}{2}$-stable. In fact, the class of self-shrinkers $\Sigma$ that are $\frac{1}{2}$-stable in $B_R$ also includes those for which each component of $\Sigma \cap B_R$ is a graph, not necessarily over the same hyperplane (see Lemma \ref{lem:gstable} and the attached remarks for these facts). %
Note also that the maximum principle for the mean curvature flow implies that for $R>\sqrt{2n}$, the intersection $\Sigma \cap B_R$ cannot be empty since the sphere of radius $\sqrt{2n}$ is a self-shrinker.

Theorem \ref{thm:main} represents quite a strong sense of rigidity for the hyperplane, since we recover the hyperplane by assuming only the stability condition ($\ddagger$) on a large, but compact, set. The form of this result is inspired by the results of Colding-Ilmanen-Minicozzi \cite{CM5}, which say that cylinders $\mf{S}^k\times \mathbf{R}^{n-k}$ are rigid amongst self-shrinkers that, on large balls, are mean convex with bounded curvature.

It is natural to seek such rigidity theorems for the hyperplane in the class of self-shrinkers, both since self-shrinkers model singularities of the mean curvature flow and since they are minimal with respect to a conformal metric. One of the most basic such results is a well-known consequence of Brakke's theorem (\cite{brakke}; see also \cite{white}), namely that there is an $\varepsilon>0$ such that any complete self-shrinker with entropy at most $1+\varepsilon$ must be a hyperplane.

In the theory of minimal hypersurfaces, Bernstein-type theorems are well-studied. For self-shrinkers, the first Bernstein-type result was due to Ecker and Huisken \cite{EH}, who showed that an entire self-shrinking graph with polynomial volume growth must be a hyperplane. Later, L. Wang \cite{WL} was able to remove the volume growth assumption.

More recently, Bernstein and Wang \cite[Corollary 3.2]{BW} proved that a smooth self-shrinker, whose area and Gaussian area approximate the respective quantities for the hyperplane closely in a large ball, must be a hyperplane. Actually, by the argument of Lemma 3.3 in the same paper, it can be shown that there are no $\delta$-stable self-shrinkers in large enough balls, for any $\delta < \frac{1}{2}$. In this context our result may be seen as completing the critical case $\delta = \frac{1}{2}$.

Our argument depends crucially on the following curvature estimate, which may be of independent interest:

\begin{theorem}\label{thm:curv}
Given $2\leq n\leq 6$ and $\lambda_0$, there exists $C=C(n,\lambda_0)$ so that for any \sk\ $\Sigma^n \subset \mf{R}^{n+1}$ with entropy $\lambda(\Sigma)\leq \lambda_0$ satisfying
\begin{itemize}
\item $\Sigma$ is $\frac{1}{2}$-stable in $B_R$ for $R>2$,
\end{itemize}
we have
\begin{equation}
   |A|(x) \leq C (1+|x|),  \,\,\,\,\,\,\,  \text{  for all } x \in B_{R-1}\cap \Sigma.
\end{equation}
\end{theorem}

The $\frac{1}{2}$-stability here may be replaced by $\delta$-stability for a fixed $\delta$ (see Remark \ref{rmk:ssdelta}), in particular, this includes the strictly mean convex case $H>0$.

Let us now outline the structure and main ideas of this paper. The proof of Theorem \ref{thm:main} relies on a bootstrapping argument for the curvature $|A|$.

In Section \ref{sec:key}, we show that for any $n$, the $\frac{1}{2}$-stability implies a total curvature bound that decays rapidly with $R$. Using a Simons-type inequality for self-shrinkers, we then show that if we are given a pointwise estimate for $|A|$ on $B_R$, then on a slightly smaller ball we may bound $|A|^2$ by its mean value on balls of radius $1/R$, hence by the total curvature bound. If the pointwise estimate for $|A|$ is reasonable (at most exponential in $R$), then the rapid decay of the total curvature takes over and we obtain that $|A|$ is uniformly small for large enough $R$. A compactness theorem for self-shrinkers from \cite{CM5}, together with the uniqueness of the hyperplane for complete self-shrinkers with $|A|^2 < 1/2$ (cf. \cite{LS,CL}), would then complete the proof of Theorem \ref{thm:main}.

We are able to provide the desired curvature estimate for $n\leq 6$ as Theorem \ref{thm:curv}, the proof of which is given in Section \ref{sec:curv}. We first give a self-contained proof in the cases $n\leq 5$ that relies on a Choi-Schoen \cite{CS} type estimate, which gives a pointwise bound for $|A|$ so long as its $L^n$ norm is small on small balls. We can satisfy this small energy hypothesis by adapting the techniques of \cite{SSY} to improve the $L^2$ control provided by the $\frac{1}{2}$-stability to $L^n$ control, so long as $n\leq 5$. We then sketch how to obtain Theorem \ref{thm:curv} for the full case $n\leq 6$ from (minor modifications of) the Schoen-Simon theory \cite{SS}. The key is to apply the theory at the correct scale $1/R$, at which the Gaussian area functional uniformly approximates the Euclidean area functional.

It is worth noting that for $n=2$, it can be shown that $\frac{1}{2}$-stable self-shrinkers in $\mathbf{R}^3$ have (local) quadratic area growth. As a consequence, the constant $C$ in Theorem \ref{thm:curv} need not depend on the entropy bound $\lambda_0$ (see Remark \ref{rmk:area} for more details).

The methods we have developed in Section \ref{sec:curv} can also be applied to certain other classes of hypersurfaces. A relevant example is the class of translating solitons of the mean curvature flow, or translators for short. Translators also arise as critical points for a weighted area functional, whose second variation operator we denote by $\mathfrak{L}$. In the Section \ref{sec:translators}, we sketch how our methods may be used to obtain uniform curvature estimates for $\mathfrak{L}$-stable translating solitons of the mean curvature flow.  Specifically, we have the following:
\begin{theorem}\label{thm: translator}
Given $n \leq 5$ and $\lambda_0$, there exists $C=C(n,\lambda_0)$ so that if $\Sigma^n \subset \mf{R}^{n+1}$ is an $\mathfrak{L}$-stable translator satisfying  $\text{Vol}(\Sigma\cap B_r(x)) \leq \lambda_0 r^n$  for all $x\in \mf{R}^{n+1}$ and $r>0$, then
\begin{equation}
|A|(x)\leq C
\end{equation}
for all $ x\in \Sigma$.
\end{theorem}

As in the shrinker case, it can be shown that $\mathfrak{L}$-stable translators in $\mathbf{R}^3$ have quadratic area growth, and hence for $n=2$ the constant $C$ above need not depend on $\lambda_0$.

\subsection*{Acknowledgements}

The authors would like to thank Professor William Minicozzi for his valuable advice and constant support. The second author is supported in part by the National Science Foundation under grant DMS-1308244. 

\section{Notation and Background}
\label{sec:background}

As mentioned, we will assume $n\geq 2$ throughout this paper.

\subsection{Geometry of self-shrinkers}

Let $\Sigma^n \subset \mathbf{R}^{n+1}$ be a smooth hypersurface, $\Delta$ its Laplace operator, $A$ its second fundamental form and $H$=div$_{\Sigma}\mathbf{n}$ its mean curvature. We denote by $B_R(x)$ the (closed) ball in $\mathbf{R}^{n+1}$ of radius $R$ centered at $x$. For convenience we will introduce the shorter notation $B_R = B_R(0)$.

For any hypersurface $\Sigma^n \subset \mathbf{R}^{n+1}$, the $F$-functional of $\Sigma$ is defined by
\begin{equation}
F(\Sigma)=(4\pi )^{-\frac{n}{2}} \int_{\Sigma} {\ee^{-\frac{|x|^2}{4}}} d\mu.
\end{equation}
The first variation of the $F$-functional implies that the self-shrinkers are precisely the critical points of $F$. Therefore, any \sk\ $\Sigma^n$ can be viewed as a minimal hypersurface with respect to the conformal metric $g_{ij}=\ee^{-\frac{|x|^2}{2n}} \delta_{ij}$ on $\mathbf{R}^{n+1}$.

Taking the second variation gives the stability operator $L$ from \cite{CM1} defined by
\begin{equation}
L=\Delta-\frac{1}{2}\lb{x, \nabla \cdot} +|A|^2 + \frac{1}{2}.
\end{equation}
We also use the drift operator $\mathcal{L}=\Delta-\frac{1}{2}\lb{x, \nabla \cdot}$.

The next lemma records three useful identities from \cite{CM1}.
\begin{lemma}[\cite{CM1}]\label{lemma:simons}
If $\Sigma^n \subset \mf{R}^{n+1}$ is a smooth \sk, then for any constant vector $v\in\mathbf{R}^{n+1}$ we have
\begin{equation}
L \lb{v,\mf{n}}=\frac{1}{2} \lb{v, \mf{n}},\end{equation}\begin{equation}    L H= H,
\end{equation}
and
\begin{equation}
\mathcal{L} |A|^2= |A|^2 - 2|A|^4 +2 |\nabla A|^2.
\end{equation}
\end{lemma}

Colding and Minicozzi \cite{CM1} introduced the entropy $\lambda$ of a \hps\ $\Sigma$, defined as
\begin{equation}
\lambda (\Sigma) = \sup_{x_0,t_0} F_{x_0,t_0}(\Sigma) = \sup_{x_0,t_0}\  (4\pi t_0)^{-\frac{n}{2}} \int_{\Sigma} {\ee^{-\frac{|x-x_0|^2}{4t_0}}} d\mu,
\end{equation}
where the supremum is taking over all $t_0 >0$ and $x_0 \in \mathbf{R}^{n+1}$. It was proven in \cite{CM1} that for a self-shrinker, the entropy is achieved by the $F$-functional $F_{0,1}$, so no supremum is needed. Cheng and Zhou \cite{CZ} (see also  \cite{DX3}) proved the equivalence of finite entropy, Euclidean volume growth and properness of \sks. In particular, if $\Sigma^n$ is a \sk\ with entropy $\lambda (\Sigma) \leq \lambda_0$, then for any $p$ and $r>0$, we have
\begin{equation}\label{eq:vol}
\vol (B_r (p) \cap \Sigma) \leq \ee^{-\frac{1}{4}} \int_{B_r (p) \cap \Sigma} \ee^{-\frac{|x-p|^2}{4r^2}} \leq \ee^{-\frac{1}{4}}  (4\pi)^\frac{n}{2} \lambda_0 \, r^n.
\end{equation}

Throughout, we will set $\rho = \ee^{-|x|^2/4}$.

\subsection{Stability}
\label{sec:stability}
We introduce the notion of \emph{$\delta$-stability} for \sks\ and prove a stability type inequality for graphical \sks.

\begin{definition}\label{def:stable}
Given $\delta$, we will say that a \sk\ $\Sigma$ is \emph{$\delta$-stable} in a domain $\Omega$ if
\begin{equation}\label{eq:delta-stable}
\int_\Sigma (-\phi L\phi) \rho  + \delta \int_\Sigma \phi^2 \rho \geq 0
\end{equation}
for any compactly supported function $\phi$ in $\Omega$.
\end{definition}
Note that integrating by parts gives that (\ref{eq:delta-stable}) is equivalent to
\begin{equation}\label{eq:delta-stable2}
 \int_\Sigma \Big(|A|^2+\frac{1}{2}-\delta\Big) \phi^2 \rho \leq \int_\Sigma |\nabla \phi|^2 \rho.
\end{equation}

In particular, when $\delta=0$, our $0$-stability is just the $L$-stability defined in \cite{CM2}. The next lemma shows that if a \sk\ is graphical in $B_R$, then it is $\frac{1}{2}$-stable in $B_R$. The proof is essentially same as Lemma 2.1 in \cite{WL}. For convenience of the reader, we also include a proof here.
\begin{lemma}
\label{lem:gstable}
Suppose $\Sigma^n$ is \sk\ which is graphical in $B_R$, then it is $\frac{1}{2}$-stable in $B_R$, i.e., for any compactly supported function $\phi$ in $B_R$, we have
\begin{equation}\label{eq:stable}
\int_\Sigma |A|^2 \phi^2 \rho \leq \int_\Sigma |\nabla \phi|^2 \rho.
\end{equation}

\end{lemma}
\begin{proof}
Since $\Sigma$ is graphical in $B_R$, we can find a constant vector $v$ such that $w(x)=\lb{v,\mf{n}(x)}$ is positive on $B_R \cap \Sigma$. Lemma \ref{lemma:simons} gives that $Lw=\frac{1}{2}w$. Hence, the function $h=\log w$ is well-defined, and it follows that $h$ satisfies the equation
\begin{equation}\label{eq:h=log w}
\mathcal{L}h=-|\nabla h|^2 -|A|^2.
\end{equation}
For any compactly supported function $\phi$ in $B_R$, multiplying by $\phi^2 \rho$ on both sides of (\ref{eq:h=log w}) and integrating by parts, we then have
\begin{equation}
\int_\Sigma (|A|^2+|\nabla h|^2 ) \phi^2 \rho =-\int_\Sigma (\phi^2 \mathcal{L}h)\rho =\int_\Sigma 2\phi \lb{\nabla \phi, \nabla h}\rho.
\end{equation}
Combining this with the inequality $2\phi \lb{\nabla \phi, \nabla h}\leq \phi^2 |\nabla h|^2 +|\nabla \phi|^2$ gives
\begin{equation}
\int_\Sigma |A|^2 \phi^2 \rho \leq \int_\Sigma |\nabla \phi|^2 \rho.
\end{equation}

\end{proof}

\begin{remark}
Similar to the above and to the proof of Lemma 9.15 in \cite{CM1}, we can prove that if $\Sigma$ is a self-shrinker and $\psi$ is a nontrivial function on $\Sigma$ with $L\psi = \mu \psi$, and $\psi>0$ on $\Sigma \cap B_R$, then $\Sigma$ is $\mu$-stable on $B_R$.

In particular, a self-shrinker with $H>0$ on $B_R\cap \Sigma$ is $1$-stable in $B_R$. Obviously, $0$-stable is stronger than $\frac{1}{2}$-stable and $1$-stable.
\end{remark}

\begin{remark}
By performing the integration by parts on each connected component $\Sigma_i$ of $\Sigma \cap B_R$ separately, we see that the conclusion of Lemma \ref{lem:gstable} holds under the weaker assumption that each $\Sigma_i$ is graphical, that is, if there are some $v_i$ so that $w_i = \langle v_i ,\mathbf{n}\rangle>0$ on each $\Sigma_i$.
\end{remark}

We will often refer the inequality (\ref{eq:stable}) as \emph{the stability inequality}.

In proofs we will often allow a constant $C$ to change from line to line; nevertheless $C$ will always depend only on the parameters as stated in the respective theorems. To emphasize particular cases where constants differ we will use primes ($C',C''$).

\section{Key ingredients and proof of the main theorem}
\label{sec:key}

In this section, we will prove the main theorem, that is, Theorem \ref{thm:main}. The key ingredients are the pointwise curvature estimate in Theorem \ref{thm:curv} and a rapidly decaying integral curvature estimate.

\subsection{Initial curvature estimates}

First, we use the rapid decay of the weight $\rho = \ee^{-|x|^2/4}$ to show that shrinkers which are $\frac{1}{2}$-stable on large balls $B_R$ satisfy an integral curvature estimate that decays exponentially in $R$.

\begin{proposition}\label{thm:intcurv}

Given $n$ and $\lambda_0$, there exists $C=C(n,\lambda_0)$ so that if $\Sigma^n\subset \mathbf{R}^{n+1}$ is a \sk\ with entropy $\lambda(\Sigma)\leq \lambda_0$, which is $\frac{1}{2}$-stable in $B_R$ for some $R>1$, then we have
\begin{equation} \int_{B_{R-1}\cap \Sigma} |A|^2 \leq CR^n \ee^{-\frac{R}{4}}.
\end{equation}
\end{proposition}
\begin{proof}
Let $a>0$. We may choose a smooth cutoff function $\phi$ so that $\phi \equiv 1$ on $B_{R-a}$, $\phi \equiv 0$ outside $B_R$ and $|\nabla \phi| \leq \frac{2}{a}$. From the stability inequality (\ref{eq:stable}), since $|\nabla \phi|$ is supported in $B_R\setminus B_{R-a}$, we get
\begin{equation} \int_{B_{R-a}\cap\Sigma} |A|^2 \rho \leq \frac{4}{a^2} \ee^{-\frac{1}{4}(R-a)^2} \vol(B_R\cap \Sigma)\leq \frac{C}{a^2} R^n \ee^{-\frac{1}{4}(R-a)^2},\end{equation} where we have used the volume estimate (\ref{eq:vol}) for the second inequality.

Therefore \begin{equation} \int_{B_{R-2a}\cap \Sigma} |A|^2 \leq \ee^{\frac{1}{4}(R-2a)^2} \int_{B_{R-a}\cap \Sigma}|A|^2\rho \leq \frac{C'}{a^2} R^n \ee^{-\frac{a}{2}R}.\end{equation} Taking $a=\frac{1}{2}$ gives the result.
\end{proof}

\subsection{Improvement of curvature estimates}

Using the pointwise curvature estimate Theorem \ref{thm:curv} together with the very tight integral estimate Proposition \ref{thm:intcurv}, we are now able to improve the pointwise estimate and show that the curvature is in fact uniformly small, so long as $R$ is sufficiently large.

\begin{theorem}
\label{thm:bootstrap}
Given $n,\lambda_0,C$ and $\delta>0$, there exists $R_0 = R_0(n,\lambda,C,\delta)$ such that if $R\geq R_0>2$ and $\Sigma^n \subset \mathbf{R}^{n+1}$ is a \sk\ with entropy $\lambda(\Sigma)\leq \lambda_0$, which is $\frac{1}{2}$-stable in $B_R$ and satisfies
\begin{equation}\label{eq:bootstrap 1}
   |A|(x) \leq C (1+|x|),  \,\,\,\,\,\,\,  \text{  for all } x \in B_{R-1}\cap \Sigma,
\end{equation}
then in fact \begin{equation}
   |A|(x) \leq \delta,  \,\,\,\,\,\,\,  \text{  for all } x \in B_{R-2}\cap \Sigma.
\end{equation}
\end{theorem}

\begin{proof}
We will use the Simons-type inequality for self-shrinkers. By Lemma \ref{lemma:simons}, we have
\begin{equation}\label{eq:simons}
\begin{split}
\Delta |A|^2 & = \frac{1}{2} \lb{x, \nabla |A|^2} + 2 \Big( \frac{1}{2}-|A|^2 \Big) |A| ^2 + 2|\nabla A|^2 \\ & \geq -\frac{1}{8}|x|^2  |A|^2 - 2|\nabla |A||^2 + 2\Big(\frac{1}{2} - |A|^2\Big) |A|^2 + 2 |\nabla A|^2 \\ & \geq - \frac{1}{8} |x|^2  |A|^2 + |A|^2 - 2|A|^4.
\end{split}
\end{equation}
The assumed curvature estimate (\ref{eq:bootstrap 1}) allows us to estimate part of the $|A|^4$ term, turning the above into a linear differential inequality. Specifically, for $x\in B_{R-1} \cap \Sigma$ we will have
 \begin{equation} \label{bootstrap:simons}\Delta |A|^2 \geq - \frac{R^2}{8} |A|^2 - 2CR^2|A|^2 = -C' R^2|A|^2.
\end{equation}
This will allow us to use the mean value inequality as follows. Fix $x_0 \in B_{R-2}\cap \Sigma$ and set \begin{equation}g(s) = s^{-n} \int_{B_s(x_0)\cap\Sigma}|A|^2.
\end{equation}
Using a mean value inequality for general hypersurfaces (See Lemma 1.18 in \cite{CM4}), we have the following inequality:
\begin{equation}\label{bootstrap:mean value}
g'(s) \geq \frac{1}{2s^{n+1}}\int_{B_s(x_0)\cap\Sigma} (s^2-|x-x_0|^2)\Delta |A|^2 - \frac{1}{s^{n+1}}\int_{B_s(x_0)\cap\Sigma} |A|^2 \lb{x-x_0, H\mathbf{n}}.
\end{equation}

By the shrinker equation (\ref{eq:shrinker}) we have $|H| \leq \frac{1}{2}|x| \leq \frac{R}{2}$, so together with (\ref{bootstrap:simons}) we then obtain \begin{equation} g'(s) \geq -C'R^2 s^{1-n} \int_{B_s(x_0)\cap\Sigma} |A|^2 - \frac{R}{2} s^{-n} \int_{B_s(x_0)\cap\Sigma}|A|^2 = - C'R^2 sg(s) - \frac{R}{2}g(s).
\end{equation}
Therefore the quantity
\begin{equation} g(s) \exp\left( C'R^2 s^2 + \frac{R}{2}s \right)\end{equation} is nondecreasing in $s$. Applying this monotonicity at scale $s= R^{-1} \leq 1$ and using the integral estimate Proposition \ref{thm:intcurv} gives \begin{equation} |A|^2 (x_0) \leq \frac{\ee^{C'+\frac{1}{2}}}{\omega_n}  R^n \int_{B_{\frac{1}{R}}(x_0)\cap \Sigma} |A|^2 \leq C'' R^n \int_{B_R \cap \Sigma} |A|^2 \leq C R^{2n} \ee^{-\frac{R}{4}}.\end{equation}
Here $\omega_n$ is the volume of the unit ball in $\mathbf{R}^n$.

The exponential factor decays faster than any polynomial factor, so taking $R$ large enough gives the desired result.
\end{proof}

\begin{remark}
\label{rmk:bootstrap}
Note that for $n\leq 6$, the curvature hypothesis (\ref{eq:bootstrap 1}) is automatically satisfied by Theorem \ref{thm:curv}.
Moreover, it is evident from the proof of Proposition \ref{thm:intcurv} that the conclusions of Theorem \ref{thm:bootstrap} hold with a somewhat weaker pointwise curvature estimate $|A|(x) \leq f(|x|)$, so long as this estimate is uniform in $\Sigma$ and $f = O(e^{aR})$ for some $a>0$. However, the main difficulty, even in the graphical case, seems to be in obtaining any initial \textit{pointwise} curvature estimate.
\end{remark}

\subsection{Proof of the main theorem}

To finish the proof of Theorem \ref{thm:main}, we will make use of the following smooth version of the compactness theorem for \sks\ in \cite{CM5}.

\begin{lemma}[\cite{CM5}]\label{lemma:compact}
  Let $\Sigma_i \subset \mf{R}^{n+1}$ be a
  sequence of smooth \sks\ with $\lambda (\Sigma_i) \leq \lambda_0$ and
  \begin{align}
   |A| \leq C \,\,\,\text{ on }\, B_{R_i} \cap \Sigma_i , ,
  \end{align}
  where $R_i \to \infty$.  Then there exists a subsequence $\Sigma_i'$
  that converges smoothly and with multiplicity one to a complete
  embedded \sk\ $\Sigma$ with $|A| \leq C$ and
  \begin{align}  \lim_{i \to \infty} \, \lambda
    (\Sigma_i') = \lambda (\Sigma) \, .
  \end{align}
\end{lemma}

\begin{proof}[Proof of Theorem \ref{thm:main}]
Given $n\leq 6$ and $\lambda_0$, Theorem \ref{thm:curv} gives a constant $C=C(n,\lambda_0)$ that controls the linear growth of $|A|$. Moreover, Theorem \ref{thm:bootstrap} enables us to make $|A|$ as small as we want. In particular, we can choose $R_0$ such that if $\Sigma^n$ is graphical in $B_R$ for $R\geq R_0$, then $|A|\leq 1/2$ for all $x\in B_{R-2}\cap \Sigma$.

Now we claim that there exists a constant $R \geq R_0$ such that Theorem \ref{thm:main} holds. Otherwise, there is a sequence of smooth, complete, \textit{non-flat} \sks\ $\Sigma_i \subset \mf{R}^{n+1}$ ($\Sigma_i \neq \mf{R}^n$) with $\lambda (\Sigma_i) \leq \lambda_0$ and $\Sigma_i$ is graphical in $B_{R_i}$ for $R_i \to \infty$ and $R_i\geq R_0$. Theorem \ref{thm:bootstrap} gives that $|A|\leq 1/2$ on $B_{R_i-2}\cap \Sigma_i$. Applying the compactness of Lemma \ref{lemma:compact}, there is a subsequence $\Sigma_i'$ that converges smoothly and with multiplicity one to a complete embedded \sk\ $\Sigma$ with $|A| \leq 1/2$ and

\begin{equation*}  \lim_{i \to \infty} \, \lambda
    (\Sigma_i') = \lambda (\Sigma) \, .
\end{equation*}

  Recall that any smooth complete \sk\ with $|A|^2< 1/2 $ is a hyperplane, so in particular the limit $\Sigma$ must be a hyperplane. By the smooth convergence or the convergence in entropy, we know by Brakke's theorem that for sufficiently large $i$, the self-shrinker $\Sigma_i$ must be a hyperplane. This provides the desired contradiction, completing the proof.
\end{proof}

\section{Curvature estimates for almost stable \sks}
\label{sec:curv}

In this section we will prove the curvature estimate Theorem \ref{thm:curv} for self-shrinkers $\Sigma^n$, which satisfy the stability inequality (\ref{eq:stable}). We give a self-contained proof for the cases $n\leq 5$ based on the Schoen-Simon-Yau \cite{SSY} curvature estimates for stable minimal hypersurfaces. We will also sketch how the result follows for $n\leq 6$ from the arguments of Schoen-Simon \cite{SS}.

\subsection{Schoen-Simon-Yau type estimates}

First, we prove a small energy curvature estimate of Choi-Schoen \cite{CS} type, allowing us to obtain pointwise estimates on $|A|$ from suitable integral estimates.

\begin{theorem}\label{thm:cs}
There exists $\varepsilon=\varepsilon (n)>0$ so that if $\Sigma^n \subset \mf{R}^{n+1}$ is self-shrinker, properly embedded in $B_{r_0}(x_0)$ for some $r_0 \leq \theta=\min\{1,|x_0|^{-1}\}$, which satisfies
\begin{equation}
\int_{B_{r_0}(x_0)\cap \Sigma} |A|^n < \varepsilon,
\end{equation}
then for all $0<\sigma \leq r_0$ and $y\in B_{r_0-\sigma}(x_0) \cap \Sigma$,
\begin{equation}
\sigma^2\abs{A}^2(y)\leq 1.
\end{equation}
\end{theorem}

\begin{proof}
We will follow the Choi-Schoen type argument and argue by contradiction.

Consider the function $f$ defined on $B_{r_0}(x_0)\cap \Sigma$ by   $f(x)=(r_0 - r(x))^2|A|^2 (x)$, where $r(x)=|x-x_0|$. This function vanishes on $\partial B_{r_0}(x_0)$, so it achieves its maximum at some $y_0 \in B_{r_0}(x_0)\cap \Sigma$. If $f(y_0) \leq 1$, then we are done. So we assume $f(y_0) \geq  1$ and will show that this leads to a contradiction for $\varepsilon$ sufficiently small.

Choose $\sigma>0$ so that
\begin{equation}
\sigma^2  |A|^2(y_0)=\frac{1}{4}.
\end{equation}
Then $f(y_0) \geq 1$ implies that $2 \sigma \leq r_0-r(y_0)$. Using this bound for $\sigma$ we see that $B_{\sigma} (y_0) \subset B_{r_0}(x_0)$, and the triangle inequality gives that for $x\in B_\sigma (y_0)$, we have
\begin{equation}\label{cs:5}
\frac{1}{2}\leq \frac{r_0-r(x)}{r_0-r(y_0)}\leq \frac{3}{2}.
\end{equation}
Combining (\ref{cs:5}) with the fact that $f$ achieves its maximum at $y_0$, we get that
\begin{equation}
(r_0-r(y_0))^2\sup_{B_\sigma(y_0) \cap \Sigma} |A|^2 \leq 4\sup_{B_\sigma(y_0) \cap \Sigma} f=4f(y_0)=4(r_0-r(y_0))^2 |A|^2(y_0).
\end{equation}
This gives the following estimate
\begin{equation}\label{cs:7}
\sup_{B_\sigma(y_0)  \cap \Sigma} |A|^2 \leq 4 |A|^2(y_0)=\frac{1}{\sigma^2} .
\end{equation}
By the Simons-type inequality (\ref{eq:simons}), it follows from (\ref{cs:7}) that on $B_\sigma (y_0)\cap \Sigma$,
\begin{equation}
\Delta |A|^2 \geq -\frac{1}{8} |x|^2 |A|^2 - \frac{2}{\sigma^2} |A|^2.
\end{equation}
A simple computation then gives
\begin{equation}\label{cs:10}
\Delta |A|^n \geq -\frac{n}{2} \Big(\frac{1}{8} |x|^2 + \frac{2}{\sigma^2}\Big) |A|^n.
\end{equation}
Next, for $0<s\leq \sigma$, we define the function $g(s)$ by
\begin{equation*}
g(s)=\frac{1}{s^n}\int_{\Sigma_{y_0,s}} |A|^n,
\end{equation*}
where we use $\Sigma_{y_0,s}$ to denote $B_{s}(y_0)\cap\Sigma$.

Again using the mean value inequality (\ref{bootstrap:mean value}), we also have
\begin{equation}
g'(s) \geq \frac{1}{2s^{n+1}}\int_{\Sigma_{y_0,s}} (s^2-|x-y_0|^2)\Delta |A|^n - \frac{1}{s^{n+1}}\int_{\Sigma_{y_0,s}} |A|^n \lb{x-y_0, H\mathbf{n}}.
\end{equation}
Combining this with (\ref{cs:10}) gives
\begin{equation}
\begin{split}
g'(s) & \geq -\frac{1}{2s^{n-1}} \int_{\Sigma_{y_0,s}}  \frac{n}{2} \Big(\frac{1}{8} |x|^2 + \frac{2}{\sigma^2}\Big) |A|^n  - \frac{1}{s^n} \frac{\sqrt{n}}{\sigma}\int_{\Sigma_{y_0,s}} |A|^n   \\ & \geq  -\frac{n}{4} \Big( \frac{1}{8} (|x_0| + 1)^2 + \frac{2}{\sigma^2} \Big) s g(s) - \frac{\sqrt{n}}{\sigma} g(s),
\end{split}
\end{equation}
where we use that $|x|\leq (|x_0|+1)$ and $|H| \leq \sqrt{n}|A| \leq \frac{\sqrt{n}}{\sigma}$ for $x\in B_{\sigma}(y_0)\cap \Sigma$.

It follows that the  function
\begin{equation}
h(t) = g(t) \exp \Big\{ \Big(\frac{n}{4} \Big( \frac{(|x_0|+1)^2 }{8} + \frac{2}{\sigma^2} \Big)\frac{t^2}{2} +\frac{\sqrt{n}}{\sigma} t \Big\}
\end{equation}
is non-decreasing for $0<t \leq \sigma$.

Applying at $t=\sigma$, since $\sigma \leq r_0 \leq \min\{1,|x_0|^{-1}\}$, we get
\begin{equation}
\frac{\omega_n}{2^n\sigma^n} = \omega_n |A|^n (y_0) = h(0) \leq h(\sigma) \leq \ee^{\frac{5n}{16}+\sqrt{n}} \frac{1}{\sigma^n} \int_{B_\sigma(y_0) \cap \Sigma} |A|^n \leq \frac{\ee^{2n}}{\sigma^n} \varepsilon,
\end{equation}
where $\omega_n$ is the volume of the unit ball in $\mf{R}^n$.

This gives a contradiction for sufficiently small $\varepsilon$.
\end{proof}

Theorem \ref{thm:cs} holds in any dimension, but of course the hypotheses require an $L^n$ bound on $|A|$, and at face value the stability inequality (\ref{eq:stable}) only provides an $L^2$ bound. So we now adapt the arguments of Schoen-Simon-Yau \cite{SSY} to improve our control on $|A|$.

\begin{theorem}\label{thm:ssy}
Suppose that $\Sigma^n\subset \mf{R}^{n+1}$ is a smooth \sk\ which is $\frac{1}{2}$-stable in $B_R$, and let $\phi$ be a smooth function with compact support in $B_R$. Then for all $q\in [0,\sqrt{2/(n+1)})$, we have
\begin{equation}
\int_\Sigma |A|^{4+2q} \phi^2 \ee^{-\frac{|x|^2}{4}}  \leq C\int_\Sigma |A|^{2+2q} |\nabla \phi|^2 \ee^{-\frac{|x|^2}{4}}  + CR^2 \int_\Sigma |A|^{2+2q} \phi^2 \ee^{-\frac{|x|^2}{4}}  ,
\end{equation}
for some $C=C(n,q)$.
\end{theorem}
\begin{proof}
We apply the stability inequality (\ref{eq:stable}) with test function $|A|^{1+q}\phi$. This gives
\begin{equation}\label{ssy:2}
\begin{split}
\int_\Sigma |A|^{4 + 2q} \phi^2 \rho &  \leq \int_\Sigma |A|^{2 + 2q}  |\nabla \phi|^2 \rho  + (1 + q)^2 \int_\Sigma |A|^{2q} \phi^2 |\nabla |A||^2 \rho \\ & + 2(1 + q) \int_\Sigma |A|^{1 + 2q} \phi \lb{\nabla \phi, \nabla |A|} \rho.
\end{split}
\end{equation}

From the Simons-type identity for \sks, i.e., Lemma \ref{lemma:simons}, we have
\begin{equation}\label{ssy:4}
|A|\mathcal{L}|A| = \frac{1}{2}|A|^2 - |A|^4 + |\nabla A|^2 -|\nabla |A||^2.
\end{equation}
Now we use the following inequality for general \hps s  from \cite[Lemma 10.2]{CM1}
\begin{equation}\label{ssy:5}
\Big(1 + \frac{2}{n+1} \Big) |\nabla |A||^2 \leq |\nabla A|^2 + \frac{2n}{n+1} |\nabla H|^2.
\end{equation}

The \sk\ equation implies that $|\nabla H| \leq \frac{1}{2}|x| |A| \leq \frac{R}{2} |A|$ for all $x \in B_R(0) \cap \Sigma$. Therefore,
\begin{equation}\label{ssy:6}
|A|\mathcal{L}|A| \geq -|A|^4 + \frac{2}{n+1} |\nabla |A||^2 -\frac{n}{n+1} \frac{R^2}{2} |A|^2.
\end{equation}

Multiplying by $|A|^{2q}\phi^2 \rho$ and integrating by parts, we get
\begin{equation}\label{ssy:8}
\begin{split}
\int_\Sigma |A|^{1 + 2q} \phi^2 \mathcal{L}|A| \rho & = -\int_\Sigma \lb{\nabla |A|, \nabla (|A|^{1 + 2q} \phi^2)} \rho  \\ &  = -\int_\Sigma |A|^{1 + 2q} \lb{\nabla |A|, \nabla \phi^2} \rho - (1 + 2q) \int_\Sigma |A|^{2 q}\phi^2 |\nabla |A||^2 \rho   \\  & \geq  \int_\Sigma \left( \frac{2}{n+1} |\nabla |A||^2 -\frac{n R^2}{2(n + 1)} |A|^2 - |A|^4  \right) |A|^{2 q}\phi^2 \rho.
\end{split}
\end{equation}

This implies
\begin{equation}\label{ssy:10}
\begin{split}
\frac{2}{n + 1} \int_\Sigma |\nabla |A||^2 |A|^{2q} \phi^2 \rho & \leq \int_\Sigma |A|^{4 + 2q}\phi^2 \rho + \frac{n R^2}{2 (n + 1)} \int_\Sigma |A|^{2 + 2q} \phi^2 \rho \\ & -2 \int_\Sigma |A|^{1 + 2q} \phi \lb{\nabla \phi, \nabla |A|} \rho - (1 + 2q) \int_\Sigma |A|^{2q} \phi^2 |\nabla |A||^2 \rho .
\end{split}
\end{equation}

Combining (\ref{ssy:2}) and (\ref{ssy:10}) gives
\begin{equation}
\begin{split}
\frac{2}{n+1}  \int_\Sigma |\nabla |A||^2 |A|^{2q} \phi^2 \rho &  \leq   \int_\Sigma |A|^{2 + 2q} |\nabla \phi|^2 \rho + q^2 \int_\Sigma |A|^{2q} |\nabla |A||^2 \phi^2 \rho   \\ & + 2q \int_\Sigma |A|^{1 + 2q}\phi \lb{\nabla \phi, \nabla |A|} \rho + \frac{n R^2}{2(n + 1)} \int_\Sigma |A|^{2 + 2q} \phi^2 \rho.
\end{split}
\end{equation}

Using the absorbing inequality, we obtain that $2|A|\phi \lb{\nabla \phi, \nabla |A|}\leq \frac{1}{a}|A|^2|\nabla \phi|^2 + a |\nabla |A||^2 \phi^2$.  Therefore, we get
\begin{equation}\label{ssy:14}
\begin{split}
\Big(\frac{2}{n + 1} - q^2 - a q \Big) \int_\Sigma  |\nabla |A||^2 |A|^{2q} \phi^2 \rho & \leq \Big(1 + \frac{q}{a}\Big)  \int_\Sigma |A|^{2 + 2q} |\nabla \phi|^2 \rho  \\ &  + \frac{n R^2}{2(n+1)} \int_\Sigma  |A|^{2 + 2q} \phi^2  \rho.
\end{split}
\end{equation}

Applying the Cauchy-Schwarz inequality to absorb the cross-term in (\ref{ssy:2}) and then substituting (\ref{ssy:14}) gives
\begin{equation}
\begin{split}
\int_\Sigma  |A|^{4 + 2q}  \phi^2 \rho & \leq \left( 2 + \frac{2 (1 + q)^2 (1 + \frac{q}{a})}{\frac{2}{n + 1} - q^2 -a q} \right)   \int_\Sigma |A|^{2 + 2q} |\nabla \phi|^2 \rho  \\ &  + \frac{\frac{n}{n + 1}(1 + q)^2 R^2}{\frac{2}{n + 1}- q^2 - a q}  \int_\Sigma |A|^{2 + 2q} \phi^2 \rho.
\end{split}
\end{equation}

Thus so long as $q^2 < \frac{2}{n+1}$, there exists a constant $C$ depending on $n$ and $q$  such that
\begin{equation}\label{ssy:16}
\int_\Sigma |A|^{4 + 2q} \phi^2 \rho  \leq C \int_\Sigma |A|^{2 + 2q} |\nabla \phi|^2 \rho + C R^2 \int_\Sigma |A|^{2 + 2q}\phi^2 \rho .
\end{equation}
\end{proof}

\begin{remark}\label{ssy:q=0}
In particular, if we set $q=0$, then
\begin{equation}
\begin{split}
\int_\Sigma |A|^4 \phi^2 \rho & \leq C \int_\Sigma |A|^2 |\nabla \phi|^2 \rho +  CR^2 \int_\Sigma |A|^2  \phi^2 \rho .  
\end{split}
\end{equation}

It will also be useful to record that (for $n \leq 6$) we may set $q=1/2$ to obtain
\begin{equation}
\int_\Sigma |A|^5 \phi^2 \rho  \leq C \int_\Sigma |A|^3 |\nabla \phi|^2 \rho + C R^2 \int_\Sigma |A|^3 \phi^2 \rho.
\end{equation}

\end{remark}

We now record a lemma that estimates a `scale-invariant energy' $r^{p-n} \int_{B_r(x)\cap \Sigma}|A|^p$. This lemma will be convenient for obtaining our curvature estimates for low dimensions, but in fact holds for all $n\geq 2$.

\begin{lemma}\label{lem:cvA}
Given $n\geq 2$, $\lambda_0>0$ and $2\leq p\leq 4$, there exists $C=C(n,p,\lambda_0)$ so that if $\Sigma^n \subset \mf{R}^{n+1}$ is a \sk\ with $\lambda(\Sigma) \leq \lambda_0$ and $\Sigma$ is $\frac{1}{2}$-stable in $B_R$ for some $R>2$, then for all $x_0 \in B_{R-1} \cap \Sigma$ and $r \leq \frac{1}{2}\theta$, where $\theta = \min\{1,|x_0|^{-1} \}$, we have
\begin{equation}
\int_{B_r(x_0) \cap \Sigma} |A|^p \leq C r^{n-p}.
\end{equation}
\end{lemma}
\begin{proof}
Fix $x_0 \in B_{R-1}\cap \Sigma$, and set $r(x)=|x-x_0|$. Also let $r_0 \leq \theta$. First note that with this choice of $r_0$, we have \begin{equation}\frac{\sup_{B_{r_0}(x_0)}\rho}{\inf_{B_{r_0}(x_0)} \rho} \leq \ee.\end{equation} This is clear if $|x_0|\leq 1$, since in this case any $x\in B_{r_0}(x_0)$ satisfies $|x|\leq |x_0|+r_0 \leq 2$. On the other hand, if $|x_0| \geq 1$ then $r_0 \leq |x_0|$, so $ \ee^{-\frac{1}{4}(|x_0|+r_0)^2} \leq \ee^{-\frac{1}{4}|x|^2} \leq \ee^{-\frac{1}{4}(|x_0|-r_0)^2} $ for $x\in B_{r_0}(x_0)$, and hence \begin{equation}\frac{\sup_{B_{r_0}(x_0)}\rho}{\inf_{B_{r_0}(x_0)} \rho} \leq \ee^{|x_0|r_0}\leq \ee.\end{equation}

Now we may fix a smooth cutoff function $\phi$ with $\phi=1$ if $r\leq r_0$, $\phi=0$ if $r> 2r_0$, and such that $|\nabla \phi| \leq \frac{2}{r_0}$.

From the stability inequality (\ref{eq:stable}) and the above discussion we have
\begin{equation}
\label{cvA:p=2}
\begin{split}
\int_{B_{r_0}(x_0)\cap \Sigma} |A|^2 &\leq \frac{1}{\inf_{B_{r_0}(x_0)} \rho}  \int_{ \Sigma} |A|^2\phi^2 \rho \leq \frac{\sup_{B_{r_0}(x_0)}\rho}{\inf_{B_{r_0}(x_0)}\rho} \int_{\Sigma} |\nabla \phi|^2  \\& \leq \frac{4\ee}{r_0^2} \vol(B_{2r_0}(x_0) \cap \Sigma) \leq C r_0^{n-2},
\end{split}\end{equation}
since $r_0 \leq \frac{1}{|x_0|}$. Again we have used the volume bound (\ref{eq:vol}) for the last inequality. This completes the proof for $p=2$.

Now further suppose $r_0 \leq \frac{1}{2}\theta$. Arguing as above using Remark \ref{ssy:q=0} ($q=0$), we obtain
\begin{equation}
\label{cvA:p=4}
\int_{B_{r_0}(x_0)\cap \Sigma} |A|^4 \leq  C \left( \frac{4}{r_0^2} + (|x_0|+1)^2\right)\int_{B_{2r_0}(x_0)\cap \Sigma} |A|^2 .
\end{equation}
We may now apply (\ref{cvA:p=2}) at radius $2r_0$ to conclude that \begin{equation} \int_{B_{r_0}(x_0)\cap \Sigma} |A|^4  \leq C \left( \frac{4}{r_0^2} + (|x_0|+1)^2\right) r_0^{n-2} \leq C' r_0^{n-4},\end{equation} where we have again used that $r_0 \leq \min(1,|x_0|^{-1})$. This completes the proof for $p=4$.

For $2<p<4$, we obtain the desired result by interpolating (\ref{cvA:p=2}) and (\ref{cvA:p=4}) using H\"{o}lder's inequality.
\end{proof}

We are now ready to give, in the case $n\leq 5$, the proof of the curvature estimate Theorem \ref{thm:curv}.

\begin{proof}[Proof of Theorem \ref{thm:curv} for $n\leq 5$.]
Fix $p \in B_{R-1}\cap \Sigma$, and set $r(x)=|x-p|$. Also let $r_0 \leq \frac{1}{4}\min \{1,|p|^{-1} \}$. \begin{equation}\label{curv:rho}\frac{\sup_{B_{r_0}(p)}\rho}{\inf_{B_{r_0}(p)} \rho} \leq \ee.\end{equation}

The goal is to show that $\Sigma$ has small energy at a uniform scale $\delta r_0$, in the sense that $\int_{B_{\delta r_0}(p)\cap \Sigma} |A|^n < \varepsilon$, where $\varepsilon$ is as in Theorem \ref{thm:cs} and $\delta$ depends only on $n$ and $\lambda_0$. From Theorem \ref{thm:cs} we would then conclude that \begin{equation} |A|(p) \leq \frac{1}{\delta r_0}  \leq C(1+|p|)\end{equation} as claimed.

To achieve this, we will use a logarithmic cutoff function. In particular, we fix a large integer $k$ to be determined later, and define a cutoff function $\eta$ by
\begin{equation}
\eta=\begin{cases} 1 &\mbox{if } r \leq \ee^{-k} r_0, \\
\frac{\log(r_0) - \log(r)}{k}& \mbox{if }\ee^{-k} r_0 < r \leq r_0, \\ 0 &\mbox{if } r > r_0. \end{cases}
\end{equation}
Note that $| \nabla \eta| \leq \frac{1}{k r}$ and $|\nabla \eta|$ is supported in the annulus  between $\ee^{-k} r_0$ and $r_0$.  Using (\ref{curv:rho}) as in the proof of Lemma \ref{lem:cvA}, we obtain from the stability inequality (\ref{eq:stable}) that
\begin{equation}
\label{cv2:2}
\int_{B_{\ee^{-k} r_0}(p) \cap \Sigma} |A|^2\eta^2 \leq \frac{\sup_{B_{r_0}(p)}\rho}{\inf_{B_{r_0}(p)} \rho} \int_\Sigma |\nabla \eta |^2   \leq  \frac{\ee}{k^2}\int_{(B_{r_0}(p) \setminus B_{\ee^{-k}r_0}(p)) \cap \Sigma} \frac{1}{r^2}.
\end{equation}

Since $n\geq 2$ we can use the usual trick as well as the volume estimate (\ref{eq:vol}) to bound the integral:
\begin{equation}
\begin{split}
\label{curv:n=2}
\int_{(B_{r_0}(p) \setminus B_{\ee^{-k} r_0}(p)) \cap \Sigma} \frac{1}{r^2} &= \sum_{l=0}^{k-1} \int_{(B_{\ee^{-l} r_0}(p) \setminus B_{\ee^{-l-1} r_0}(p)) \cap \Sigma} \frac{1}{r^2} \leq C \sum_{l=0}^{k-1} \ee^{2(l+1)} r_0^{-2} \ee^{-nl} r_0^n \\ & \leq C' r_0^{n-2} \sum_{l=0}^{k-1} \ee^{-(n-2)l} \leq C'kr_0^{n-2}.
\end{split}
\end{equation}

Thus \begin{equation}\int_{B_{\ee^{-k} r_0}(p) \cap \Sigma} |A|^2 \leq \int_{B_{\ee^{-k} r_0}(p) \cap \Sigma} |A|^2\eta^2 \leq \frac{C}{k} r_0^{n-2}.\end{equation} This completes the proof for $n=2$, by taking $k$ sufficiently large.

Now using the Cauchy-Schwarz inequality and that $B_{\ee^{-k}r_0}(p)\subset B_{r_0}(p)$, we have
 \begin{equation}
 \begin{split}
 \int_{B_{\ee^{-k} r_0}(p) \cap \Sigma} |A|^3\eta^2 &\leq \left(\int_{B_{\ee^{-k} r_0}(p) \cap \Sigma} |A|^2\eta^2 \right)^\frac{1}{2} \left(\int_{B_{ r_0}(p) \cap \Sigma}|A|^4\eta^2\right)^\frac{1}{2}.\end{split}\end{equation} Using (\ref{cv2:2}) for the $|A|^2$ factor and Lemma \ref{lem:cvA} for the $|A|^4$ factor we conclude that

 \begin{equation} \label{curv:n=3}\int_{B_{\ee^{-k} r_0}(p) \cap \Sigma} |A|^3\leq \int_{B_{\ee^{-k} r_0}(p) \cap \Sigma} |A|^3\eta^2 \leq \frac{C}{k} r_0^{n-3}.\end{equation} Thus, again by taking $k$ sufficiently large, the result follows for $n=3$.

Consider the remaining cases $n=4,5$. Using (\ref{curv:rho}) to estimate the weight $\rho$ as before, and applying Remark \ref{ssy:q=0}, we obtain
\begin{equation}
\begin{split}
\int_{\Sigma} |A|^n \eta^2 &  \leq  C \left(  \int_\Sigma |A|^{n-2} |\nabla \eta|^2  +  (|p|+1)^2 \int_\Sigma |A|^{n-2} \eta^2 \right).
\end{split}
\end{equation}

By (\ref{curv:n=2}) or (\ref{curv:n=3}) respectively, the second term can be estimated using that $r_0 \leq \min(1,|p|^{-1})$,
\begin{equation}
(|p|+1)^2 \int_\Sigma |A|^{n-2} \eta^2 \leq \frac{C}{k}(|p|+1)^2 r_0^{n-(n-2)} \leq \frac{C'}{k}.
\end{equation}

To handle the first term we use the logarithmic trick again:
\begin{equation}
\begin{split}
\int_\Sigma |A|^{n-2} |\nabla \eta|^2 &\leq \frac{1}{k^2} \sum_{l = 0}^{k-1} \int_{B_{\ee^{-l} r_0(p)} \setminus B_{\ee^{-l-1} r_0(p)}} |A|^{n-2} \frac{1}{ r^2} \\& \leq \frac{C}{k^2} \sum_{l = 0}^{k-1} \ee^{2(l+1)}r_0^{-2} \ee^{(n-(n-2))l}r_0^{n-(n-2)} = \frac{C'}{k} .
\end{split}
\end{equation}
Here we have used Lemma \ref{lem:cvA} to get to the second line. Thus \begin{equation}\int_{B_{\ee^{-k} r_0}(p)\cap \Sigma} |A|^n \leq \frac{C}{k}\end{equation} in the remaining cases $n=4,5$, and the proof is complete.

\end{proof}

\begin{remark}\label{rmk:area}
When $n=2$, we can prove that (small) geodesic balls of $\frac{1}{2}$-stable self-shrinkers in $\mf{R}^3$ with small radius have at most quadratic area growth, without assuming an entropy bound. Consequently, we do not need to assume an entropy bound in Theorem \ref{thm:curv}.

The key here is to use the Gauss-Bonnet theorem, together with a technical lemma due to Shiohama-Tanaka  (\cite{ST93,ST89}; see \cite{ER} for a more directly applicable statement), 
to estimate the area of geodesic balls $B^\Sigma_{r_0}(x)$ by a curvature integral:
\begin{equation}
\area(B^\Sigma_{r_0}(x)) -\pi r_0^2 \leq \frac{1}{4}\int_{B^\Sigma_{r_0}(x)} |A|^2(r_0-r)^2.
\end{equation}
This bound holds for general surfaces and any $r_0$. If $\Sigma$ is a $\frac{1}{2}$-stable self-shrinker in $\mf{R}^3$, applying the stability inequality to the right hand side at the right scale $r_0 \leq \theta = \min(1, |x|^{-1})$ yields the local area bound \begin{equation} \area(B^\Sigma_{r_0}(x))  \leq \frac{4\pi}{4-\ee}r_0^2.\end{equation}

To obtain the desired curvature estimate, we then note that the intrinsic analogue of Theorem \ref{thm:cs} - that is, with extrinsic balls $B_{r_0}(x_0)$ replaced by geodesic balls $B^\Sigma_{r_0}(x_0)$ - also holds, due to a chord-arc bound for general surfaces (see Lemma 2.4 of \cite{CM4}). The resulting intrinsic small energy hypothesis can be satisfied using the same method as in Theorem \ref{thm:curv}, using a logarithmic cutoff function defined in terms of intrinsic distances instead of extrinsic distances.
\end{remark}

\subsection{Schoen-Simon theory}

To obtain the desired curvature estimate Theorem \ref{thm:curv} in the case $n=6$, we need to apply the Schoen-Simon theory \cite{SS}, with a few minor modifications that we will outline here. Of course, this argument will indeed apply for all $2\leq n \leq 6$.

The key point is to apply the theory at the optimal scale $\theta = \min(1,|x_0|^{-1})$ (see for instance Section 12 of \cite{CM1}). At this scale, the conformal metric $g_{ij} = \ee^{-\frac{|x|^2}{2n}}\delta_{ij}$ is uniformly Euclidean in the sense that the volume form $\rho = \ee^{-\frac{|x|^2}{4}}$ satisfies the familiar estimate \begin{equation} \frac{\sup_{B_{ \theta}(x_0)}\rho}{\inf_{B_{\theta}(x_0)}\rho} \leq  \ee, \end{equation} with similar uniform estimates for its derivatives.

In particular, consider a smooth self-shrinker $\Sigma^n \subset \mathbf{R}^{n+1}$, which has entropy $\lambda(\Sigma) \leq \lambda_0$ and is $\frac{1}{2}$-stable in $B_R$, $R>1$. For $x_0 \in B_{R-1}\cap \Sigma$, we wish to apply the Schoen-Simon theory to the (renormalised) $F$-functional \begin{equation} \ee^{\frac{1}{4}|x_0|^2}\int_\Sigma \ee^{-\frac{1}{4}|x|^2}.\end{equation}

By the uniform estimates mentioned above, it may be verified that this functional satisfies all the conditions in Section 1 of \cite{SS} with parameters depending only on $n$. Also, the entropy bound gives a required density bound (\ref{eq:vol}). The main issue with this approach is that $\Sigma$ may not be stable for the $F$ functional. At the scale $\theta$, however, the $\frac{1}{2}$-stability inequality (\ref{eq:stable}) gives that \begin{equation} \int_\Sigma |A|^2 \phi^2 \leq  \frac{\sup_{B_{ \theta}(x_0)}\rho}{\inf_{B_{\theta}(x_0)}\rho} \int_\Sigma |\nabla \phi|^2 \leq \ee\int_\Sigma |\nabla \phi|^2,\end{equation} for compactly supported functions $\phi$ in $B_\theta(x_0)$. All the arguments of \cite{SS} then go through with this slightly weaker stability inequality, at the cost of carrying around the universal constant $\ee$.

In particular, the conclusions of theorem 3 of \cite{SS} hold, so since $\Sigma$ is smooth we conclude that \begin{equation} |A|(x_0) \leq C \theta^{-1} \leq C(1+|x_0|),\end{equation} where $C=C(n,\lambda_0)$ as desired.

\begin{remark}
\label{rmk:ssdelta}
Indeed, the Schoen-Simon argument goes through assuming only a $\delta$-stability inequality (\ref{eq:delta-stable2}) at the optimal scale $\theta = \min(1,|x_0|^{-1})$, for $\delta$ depending only on $n$. (That is, assuming a uniform lower bound for the least eigenvalue of $L$ on balls $B_\theta(x_0)$.)

In particular, the curvature estimate also applies to self-shrinkers with bounded entropy and positive mean curvature $H>0$ in $B_R$, $R>1$. In this strictly mean convex setting, the curvature estimate was already known to Colding-Ilmanen-Minicozzi.
\end{remark}

\section{Curvature estimates for $\mathfrak{L}$-stable translators}\label{sec:translators}

In this section, we will sketch how our methods in Section \ref{sec:curv} can be used to prove uniform curvature estimates for $\mathfrak{L}$-stable translators. We will only outline the proofs, since they are very similar to the shrinker cases.

A smooth hypersurface $\Sigma^{n} \subset \mathbf{R}^{n+1}$ is called a \textit{translating soliton}, or \textit{translator} for short, if it satisfies the equation
\begin{equation}
H=-\lb{y,\mathbf{n}}.
\end{equation}
Here $y\in \mathbf{R}^{n+1}$ is a constant vector.

Such a hypersurface $\Sigma$ evolves under mean curvature flow by translation (in the direction $y$). Moreover, translators arise as blow-up solutions of mean curvature flow at type \uppercase\expandafter{\romannumeral2} singularities.

In $\mathbf{R}^{n+1}$, there exists a unique rotationally symmetric, strictly convex translator (up to rigid motions). For $n=1$, it is known as the \textit{grim reaper}. For $n\geq 2$, it is an entire graph and is usually called the \textit{bowl soliton}. Recently, there have been some remarkable uniqueness and rigidity results for bowl solitons; see for instance \cite{WXJ}, \cite{CSS} and \cite{HR15}.

For simplicity, we may assume $y=\mathbf{e}_{n+1}$, so that translators satisfy the equation
\begin{equation}
H=-\lb{\mathbf{e}_{n+1},\mathbf{n}}.
\end{equation}

Translators are critical points of the functional $\mathcal{F}(\Sigma)=\int_\Sigma \ee^{x_{n+1}}\, d\mu$, where $x_{n+1}=\langle x, \mathbf{e}_{n+1}\rangle$ is the $(n+1)$-th coordinate of the position vector $x$. Thus, like self-shrinkers, translators can also be viewed as minimal hypersurfaces in a conformal metric. By computing the second variation
formula of the functional $\mathcal{F}$, we can define the corresponding stability operator $\mathfrak{L}$ as
\begin{equation}
\mathfrak{L}=\Delta+\lb{\mathbf{e}_{n+1},\nabla \cdot}+|A|^2.
\end{equation}

\begin{definition}\label{def:translator-stable}
Following the notion in \cite{LS2}, we say that a translator $\Sigma$ is $\mathfrak{L}$-stable, if for any compactly supported function $\phi$, we have
\begin{equation}\label{eq:translator-stable}
\int_\Sigma (-\phi \mathfrak{L}\phi )\,\ee^{x_{n+1}}\,d\mu\geq 0.
\end{equation}
\end{definition}

It was proved by Shahriyari \cite[Theorem 2.5]{LS2} that all translating graphs are $\mathfrak{L}$-stable. Note that (\ref{eq:translator-stable}) is equivalent to
\begin{equation}
\int_\Sigma |A|^2 \phi^2 \ee^{x_{n+1}} \leq \int_\Sigma |\nabla \phi |^2 \ee^{x_{n+1}}.
\end{equation}

\begin{proof}[Outline of the proof of Theorem \ref{thm: translator}]
In order to prove Theorem \ref{thm: translator}, we also need to establish a Choi-Schoen type estimate for translators, which gives a pointwise estimate for $|A|$  as long as its $L^n$ norm is small. When $n=2$, such an $L^2$ bound can be obtained from the inequality (\ref{eq:translator-stable}). For $3\leq n\leq 5$, we can again adapt the techniques of Schoen-Simon-Yau to improve the control on $|A|$. The corresponding theorems are as follows:

\begin{theorem}\label{thm:translator-cs}
There exists $\varepsilon=\varepsilon (n)>0$ so that if $\Sigma^n\subset \mathbf{R}^{n+1}$ is a properly embedded translator in $B_{r_0}(x_0)$ for $r_0 \leq 1$ which satisfies
\begin{equation}
\int_{B_{r_0}(x_0)\cap \Sigma} \abs{A}^n < \varepsilon,
\end{equation}
then for all $0<\sigma \leq r_0$ and $y\in B_{r_0-\sigma}(x_0)$,
\begin{equation}
\sigma^2\abs{A}^2(y)\leq 1.
\end{equation}
\end{theorem}

\begin{theorem}\label{thm:translator-ssy}
Suppose that $\Sigma^n \subset \mf{R}^{n+1}$ is a properly embedded $\mathfrak{L}$-stable translator. Then for all $q \in [0,\sqrt{2/(n+1)})$, we have
\begin{equation}
\int_\Sigma |A|^{4+2q} \phi^2 \ee^{x_{n+1}} \leq C \Big[ \int_\Sigma |A|^{2+2q} |\nabla \phi|^2 \ee^{x_{n+1}} +\int_\Sigma |A|^{2+2q} \phi^2 \ee^{x_{n+1}}  \Big],
\end{equation}
where $C=C(n,q)$ and $\phi$ is a smooth function with compact support.
\end{theorem}

The proofs of Theorem \ref{thm:translator-cs} and Theorem \ref{thm:translator-ssy} are similar to the shrinker cases, since translators also satisfy a Simons-type identity:
\begin{equation}\label{eq:translator-simons}
\mathfrak{L}|A|^2=2|\nabla A|^2-|A|^4.
\end{equation}
This gives a Simons-type inequality
\begin{equation}
\begin{split}
\Delta |A|^2  &\geq -\frac{1}{2}|A|^2-2\abs{\nabla |A|}^2+2\abs{\nabla A}^2-2\abs{A}^4 \\& \geq -\frac{1}{2}\abs{A}^2-2|A|^4.
\end{split}
\end{equation}
Moreover, combining (\ref{eq:translator-simons}) with (\ref{ssy:5}) gives
\begin{equation}\label{eq:trans-simons2}
|A|\Big(\Delta |A|+ \lb{\ee_{n+1},\nabla |A|} \Big) \geq -|A|^4+\frac{2}{n+1} |\nabla |A||^2 -\frac{2n}{n+1}|A|^2.
\end{equation}
Here we use that $|\nabla H|\leq |A|$ for translators.   The inequality (\ref{eq:trans-simons2}) is the essential inequality for the proof of Theorem \ref{thm:translator-ssy}.

Finally, applying suitable cutoff functions as in the proof of Theorem \ref{thm:curv} leads to the proof of Theorem \ref{thm: translator}.
\end{proof}


\bibliographystyle{alpha}
\bibliography{graphical_shrinkers}
\end{document}